\documentclass[11pt]{amsproc}

\usepackage{amsmath, amsthm, amssymb, mathtools, mathrsfs, stmaryrd}
\usepackage{ascmac}
\usepackage{comment}
\usepackage{bm}
\allowdisplaybreaks

\usepackage{graphicx}
\usepackage[top=29mm, bottom=29mm, left=24mm, right=24mm]{geometry}

\usepackage{hyperref}

\usepackage{here}
\usepackage{time}
\usepackage[abbrev]{amsrefs}

\usepackage{xcolor}
\usepackage[capitalize,nameinlink,noabbrev,nosort]{cleveref}
\hypersetup{
	colorlinks=true,       % false: boxed links; true: colored links
	linkcolor=brown,          % color of internal links
	citecolor=brown,        % color of links to bibliography
	filecolor=brown,      % color of file links
	urlcolor=brown,           % color of external links
}

\makeatletter
\@namedef{subjclassname@2020}{\textup{2020} Mathematics Subject Classification}
\makeatother

\newtheorem*{theorem*}{\hspace{-6.3mm}\textbf{Theorem}}  %%% 

\newtheorem{theoremcounter}{Theorem Counter}[section]

\theoremstyle{remark}
\newtheorem*{remark*}{Remark}

\theoremstyle{definition}

\newtheorem*{example*}{Example}

\theoremstyle{plain}

\newtheorem{proposition}[theoremcounter]{Proposition}
\newtheorem{corollary}[theoremcounter]{Corollary}

\newtheorem{theorem}[theoremcounter]{Theorem}

\numberwithin{equation}{section}

\newcommand{\Z}{\mathbb{Z}}

\newcommand{\bbH}{\mathbb{H}}

\DeclareMathOperator{\SL}{SL}

\begin{document}
% --------------------------------------------------------------------------

\title[]{Hauptmoduln and even-order mock theta functions modulo 2} 

\author[]{Soon-Yi Kang}
\address{Department of Mathematics, Kangwon National University, Chuncheon, 200-701, Republic of Korea}
\email{sy2kang@kangwon.ac.kr}

\author[]{Seonkyung Kim}
\address{Department of Mathematics, Kangwon National University, Chuncheon, 200-701, Republic of Korea}
\email{cbr9r7022@kangwon.ac.kr}

\author[]{Toshiki Matsusaka}
\address{Faculty of Mathematics, Kyushu University, Motooka 744, Nishi-ku, Fukuoka 819-0395, Japan}
\email{matsusaka@math.kyushu-u.ac.jp}

\author[]{Jaeyeong Yoo}
\address{Department of Mathematics, Kangwon National University, Chuncheon, 200-701, Republic of Korea}
\email{anamory726@kangwon.ac.kr}

\thanks{This work was supported by the MEXT Initiative through Kyushu University's Diversity and Super Global Training Program for Female and Young Faculty (SENTAN-Q). The first author was supported by the National Research Foundation of Korea (NRF) funded by the Ministry of Education (NRF-2022R1A2C1007188) and the third author was supported by JSPS KAKENHI (JP21K18141 and JP24K16901).}

%\framebox{\today \ \now}

%11P83  	Partitions; congruences and congruential restrictions
%11F30  	Fourier coefficients of automorphic forms
%11F33  	Congruences for modular and $p$-adic modular forms

% --------------------------------------------------------------------------

\maketitle

% --------------------------------------------------------------------------
\begin{abstract} 
	The Fourier coefficients $c_1(n)$ of the elliptic modular $j$-function are always even for $n \not\equiv 7 \pmod{8}$. In contrast, for $n \equiv 7 \pmod{8}$, it is conjectured that ``half" of the coefficients take odd values. In this article, we first observe in detail when $c_1(8n-1)$ is odd and show that the coefficients share the same parity as the coefficients $c_{\mu_2}(n)$ of the 2nd order mock theta function $\mu_2(q)$. Furthermore, we prove that this phenomenon also holds among several hauptmoduln and between hauptmoduln and even-order mock theta functions.
\end{abstract}
% --------------------------------------------------------------------------

%%%%%%%%%%%%%%%%%%%%%%%%%%%%%%%%%
%%%%%%%%%%%%%%%%%%%%%%%%%%%%%%%%%

% --------------------------------------------------------------------------
\section{Introduction}
% --------------------------------------------------------------------------

Let $E_4(\tau)$ and $\Delta(\tau)$ denote the weight $4$ Eisenstein series and the weight 12 cusp form for $\SL_2(\Z)$, normalized such that their leading Fourier coefficients are $1$. The elliptic modular function $j(\tau)$ is invariant on the upper half-plane $\bbH$ under the action of $\SL_2(\Z)$, and is defined by
\[
	j(\tau) \coloneqq \frac{E_4(\tau)^3}{\Delta(\tau)} = \frac{\left(1 + 240 \sum_{n=1}^\infty \left(\sum_{d \mid n} d^3\right)q^n\right)^3}{q \prod_{n=1}^\infty (1-q^n)^{24}},
\]
where $q \coloneqq e^{2\pi i\tau}$ and $\tau \in \bbH$. Its Fourier expansion is of the form
\[
	j(\tau) = \frac{1}{q} + 744 + \sum_{n=1}^\infty c_1(n) q^n = \frac{1}{q} + 744 + 196884q + 21493760q^2 + \cdots,
\]
where the coefficients $c_1(n)$ are immediately seen to be positive integers. The arithmetic properties of these coefficients have been extensively studied, particularly congruences modulo prime powers, beginning with Lehner's celebrated work~\cite{Lehner1949} and further as summarized by Alfes~\cite{Alfes2013}.

As Alfes pointed out in the introduction of her article~\cite{Alfes2013}, since the Eisenstein series $E_4(\tau)$ satisfies $E_4(\tau) \equiv 1 \pmod{2}$, it follows from
\begin{align}\label{eq:j1-2}
	j(\tau) \equiv \frac{1}{q(q)_\infty^{24}} \equiv \frac{1}{q (q^8)_\infty^3} \equiv q^{-1} + q^7 + q^{15} + q^{31} + \cdots \pmod{2}
\end{align}
that $c_1(n) \equiv 0 \pmod{2}$ holds whenever $n \not\equiv 7 \pmod{8}$. Here, we define $(a,q)_n \coloneqq \prod_{k=0}^{n-1}(1-a q^k)$ and set $(q)_n \coloneqq (q,q)_n$. In the case when $n \equiv 7 \pmod{8}$, it is conjectured that ``half" of the values of $c_1(n)$ are odd. Alfes proved that $c_1(8n-1)$ takes even and odd values infinitely often, respectively. Later, Bella\"{i}che--Green--Soundararajan~\cite{BGS2018} improved the lower bound of the number of odd values of $c_1(8n-1)$.

We examine in detail when $c_1(8n-1) \equiv 1 \pmod{2}$. The first values of $n$ up to $30$ for which this holds are as follows.
\[
	n = 1, 2, 4, 6, 7, 9, 11, 13, 14, 15, 18, 20, 23, 26, \dots.
\]
Our initial discovery is that this list can be compared with the $q$-series expansion of the 2nd order mock theta function $\mu_2(q)$ studied by McIntosh~\cite{McIntosh2007}:
\begin{align}\label{eq:mu2}
\begin{split}
	\mu_2(q) &\coloneqq \sum_{n=0}^\infty \frac{(-1)^n q^{n^2} (q,q^2)_n}{(-q^2, q^2)_n^2} = 1 + \sum_{n=1}^\infty c_{\mu_2}(n) q^n\\
		&= 1 - q + q^2 + 2q^3 - q^4 - 4q^5 + q^6 + 5q^7 - 2q^8 - 5q^9 + 4q^{10} + 7q^{11} + \cdots.
\end{split}
\end{align}
A careful observation suggests the congruence $c_1(8n-1) \equiv c_{\mu_2}(n) \pmod{2}$.

More generally, we extend our observations by comparing the parities of the Fourier coefficients of hauptmoduln and mock theta functions.
 It is known that for a congruence subgroup $\Gamma_0(N)$, the modular curve $X_0(N)$ has genus $0$ if and only if the positive integer $N$ satisfies $1 \le N \le 10, 12, 13, 16, 18, 25$. In these cases, the hauptmodul $j_N(\tau)$ serves as a generator of the field of modular functions for $\Gamma_0(N)$. According to Table 3 of Conway--Norton~\cite{ConwayNorton1979}, the hauptmoduln in these cases are given as follows. Here, we define $\eta(q) \coloneqq q^{1/24} (q)_\infty$ and set $\eta_k \coloneqq \eta(q^k)$.

\begin{figure}[H]
\centering
\begin{tabular}{c|ccccccccccccccc}
	$N$ & 1 & 2 & 3 & 4 & 5 & 6 & 7 & 8 & 9 & 10 & 12 & 13 & 16 & 18 & 25 \\ \hline
	\vspace{-3mm} \\
	$j_N$ & $j$ & $\dfrac{\eta_1^{24}}{\eta_2^{24}}$ & $\dfrac{\eta_1^{12}}{\eta_3^{12}}$ & $\dfrac{\eta_1^8}{\eta_4^8}$ & $\dfrac{\eta_1^6}{\eta_5^6}$ & $\dfrac{\eta_2^3 \eta_3^9}{\eta_1^3 \eta_6^9}$ & $\dfrac{\eta_1^4}{\eta_7^4}$ & $\dfrac{\eta_1^4 \eta_4^2}{\eta_2^2 \eta_8^4}$ & $\dfrac{\eta_1^3}{\eta_9^3}$ & $\dfrac{\eta_2 \eta_5^5}{\eta_1 \eta_{10}^5}$ & $\dfrac{\eta_4^4 \eta_6^2}{\eta_2^2 \eta_{12}^4}$ & $\dfrac{\eta_1^2}{\eta_{13}^2}$ & $\dfrac{\eta_1^2 \eta_8}{\eta_2 \eta_{16}^2}$ & $\dfrac{\eta_6 \eta_9^3}{\eta_3 \eta_{18}^3}$ & $\dfrac{\eta_1}{\eta_{25}}$
\end{tabular}
\end{figure}

For some of these hauptmoduln, Kumari--Singh~\cite{KumariSingh2019} and Ray~\cite{Ray2023} have recently investigated the parities of their Fourier coefficients. For instance, we review one of Ray's results for $N=6$. For $n \ge 1$, let $c_6(n)$ denote the $n$-th Fourier coefficient of $j_6(\tau)$, that is,
\[
	j_6(\tau) \coloneqq \frac{\eta(q^2)^3 \eta(q^3)^9}{\eta(q)^3 \eta(q^6)^9} = \frac{1}{q} + 3 + \sum_{n=1}^\infty c_6(n) q^n = \frac{1}{q} + 3 + 6q + 4q^2 - 3q^3 -12 q^4 - 8q^5 + \cdots.
\]
Ray proved the following theorem.

\begin{theorem*}[{Ray~\cite[Theorems 1.5 and 3.2]{Ray2023}}]
	The following claims hold.
	\begin{enumerate}
		\item $c_6(2n) \equiv 0 \pmod{2}$.
		\item $c_6(4n+1) \equiv 0 \pmod{2}$.
		\item $c_6(24n+11) \equiv c_6(24n+19) \equiv 0 \pmod{2}$.
		\item $c_6(24n+3)$ takes both odd and even values and satisfies that
		\[
			\lim_{X \to \infty} \frac{\#\{n \le X : c_6(24n+3) \equiv 0 \pmod{2}\}}{X} = 1.
		\]
	\end{enumerate}
\end{theorem*}
What can be observed in the remaining arithmetic progressions? For instance, $c_6(24n-1)$ takes both odd and even values. When we enumerate the values of $n$ up to $30$ that satisfy $c_6(24n-1) \equiv 1 \pmod{2}$, we obtain the following list.
\[
	n = 1, 3, 4, 5, 6, 7, 12, 13, 14, 16, 17, 18, 20, 23, 24, 29, \dots.
\]
This is comparable to the 6th order mock theta function
\begin{align}\label{eq:phi6}
\begin{split}
	\phi_6(q) &\coloneqq \sum_{n=0}^\infty \frac{(-1)^n (q, q^2)_n q^{n^2}}{(-q, q)_{2n}} = 1 + \sum_{n=1}^\infty c_{\phi_6}(n) q^n\\
		&= 1 - q + 2q^2 - q^3 + q^4 - 3q^5 + 3q^6 - 3q^7 + 4q^8 - 4q^9 + 6q^{10} - 6q^{11} + \cdots,
\end{split}
\end{align}
which was discovered in Ramanujan's Lost Notebook~\cite{AndrewsHickerson1991}, \cite[Chapter 7]{LostV}. It can be observed that  $c_6(24n-1) \equiv c_{\phi_6}(n)$, similar to the case of $j(\tau)$.

The aim of this article is to comprehensively complete the above observations regarding hauptmoduln of levels $N = 2^k, 3 \cdot 2^k, 5 \cdot 2^k$, corresponding to the even-order mock theta functions listed in~\cite[Appendix A]{BFOR2017}, specifically the 2nd, 6th, 8th, and 10th orders. For $n \ge 1$, let $c_N(n)$ denote the $n$-th Fourier coefficient of the hauptmodul $j_N(\tau)$. Then, we obtain the following theorems.

\begin{theorem}\label{thm:main-1}
	We have $c_1(n) \equiv c_2(n) \equiv c_4(n) \equiv c_8(n) \equiv c_{16}(n) \pmod{2}$ and the following:
	\begin{enumerate}
		\item If $n \not\equiv 7 \pmod{8}$, then $c_1(n) \equiv 0 \pmod{2}$.
		\item $c_1(8n-1) \equiv c_{\mu_2}(n) \equiv c_{U_0}(n) \pmod{2}$.
		\item $c_1(16n-1) \equiv c_{S_0}(n) \pmod{2}$.
		\item $c_1(16n+7) \equiv c_{S_1}(n) \pmod{2}$.
	\end{enumerate}
	Here $\mu_2(q)$ is a 2nd order mock theta function, and $U_0(q), S_0(q), S_1(q)$ are 8th order mock theta functions.
\end{theorem}

\begin{theorem}\label{thm:main-2}
	We have $c_3(n) \equiv c_6(n) \equiv c_{12}(n) \pmod{2}$ and the following:
	\begin{enumerate}
		\item $c_3(2n) \equiv c_3(4n+1) \equiv c_3(12n+7) \equiv c_3(24n+11) \equiv 0 \pmod{2}$.
		\item $c_3(24n-1) \equiv c_{f_3}(n) \equiv c_{\phi_3}(n) \equiv c_{\phi_6}(n) \equiv c_{2\mu_6}(2n) \equiv p(n) \pmod{2}$.
		\item $c_3(24n-9) \equiv c_{\psi_6}(n) \pmod{2}$.
		\item $c_3(24n+3) \equiv 1 \pmod{2}$ if and only if $n = \frac{k(k+1)}{2}$ for some $k \ge 0$.
		\item $c_3(12n+3) \equiv c_{\lambda_6}(n) \pmod{2}$.
	\end{enumerate}
	Here $f_3(q), \phi_3(q)$ are 3rd order mock theta functions, $\phi_6(q), \mu_6(q), \psi_6(q), \lambda_6(q)$ are 6th order mock theta functions, and $p(n)$ is the number of partitions of $n$.
\end{theorem}

\begin{theorem}\label{thm:main-3}
	We have $c_5(n) \equiv c_{10}(n) \pmod{2}$ and the following:
	\begin{enumerate}
		\item $c_5(2n) \equiv c_5(8n+3) \equiv c_5(40n+5+8i) \equiv c_5(40n+7+8j) \equiv 0 \pmod{2}$ for $i \in \{1,2,3,4\}$ and $j \in \{0,2\}$.
		\item $c_5(40n-1) \equiv c_{X_{10}}(n) \pmod{2}$.
		\item $c_5(40n-9) \equiv c_{\chi_{10}}(n) \pmod{2}$.
		\item $c_5(8n+1) \equiv c_5(40n+5) \pmod{2}$ and $c_5(8n+1) \equiv 1 \pmod{2}$ if and only if $n = \frac{k(k+1)}{2}$ for some $k \ge 0$.
		\item $c_5(40n+15) \equiv p_{10}(n) \pmod{2}$.
	\end{enumerate}
	Here $X_{10}(q)$ and $\chi_{10}(q)$ are 10th order mock theta functions and $p_{10}(n)$ is the number of partitions of $n$ in which no parts are multiples of $10$.
\end{theorem}

\begin{remark*}
	In proving Ray's Theorem (4), a result by Serre was applied. However, our \cref{thm:main-2} (4) provides a more precise density with a simpler and more direct proof, as shown in \cref{cor-Ray}. A similar theorem is also proved for the level 10 case in \cite[Theorem 4.1]{Ray2023}, but once again, our \cref{thm:main-3} (4) provides a refinement.
\end{remark*}

This article is organized as follows. In \cref{sec:2}, we provide the necessary list of mock theta functions. Starting from \cref{sec:3}, we prove our theorems for each case. Lastly, in \cref{sec:remarks}, we explore results for prime moduli beyond modulo 2 and present some open questions related to our work.

% --------------------------------------------------------------------------
\section{Mock theta functions}\label{sec:2}
% --------------------------------------------------------------------------

The full list of mock theta functions can be found in \cite[Appendix A]{BFOR2017} and \cite[Appendix A]{ShibukawaTsuchimi2024}. As noted in~\cite{ShibukawaTsuchimi2024}, there is an error in the definition of the 3rd order mock theta function $\rho(q)$ in the list from~\cite{BFOR2017}.

The definition of the 2nd order mock theta function $\mu_2(q)$ was already introduced in~\eqref{eq:mu2}. 

The 3rd order mock theta function $f_3(q)$ and $\phi_3(q)$ first appeared in Ramanujan's last letter, defined as
\begin{align}\label{def:f3}
	f_3(q) &\coloneqq \sum_{n=0}^\infty \frac{q^{n^2}}{(-q,q)_n^2} = 1 + q - 2q^2 + 3q^3 - 3q^4 + 3q^5 - 5q^6 + 7q^7 - 6q^8 + 6q^9 + \cdots,\\
	\phi_3(q) &\coloneqq \sum_{n=0}^\infty \frac{q^{n^2}}{(-q^2, q^2)_n} = 1 + q - q^3 + q^4 + q^5 - q^6 - q^7 + 2q^9 + \cdots.
\end{align}

The 6th order mock theta functions $\phi_6(q), \psi_6(q), \lambda_6(q)$, and $2\mu_6(q)$, from Ramanujan's lost notebook~\cite{AndrewsHickerson1991}, \cite[Chapter 7]{LostV}, are defined by \eqref{eq:phi6} and
\begin{align}
	\psi_6(q) &\coloneqq \sum_{n=0}^\infty \frac{(-1)^n q^{(n+1)^2} (q,q^2)_n}{(-q,q)_{2n+1}} = q - q^2 + q^3 - 2q^4 + 3q^5 - 2q^6 + 2q^7 - 4q^8 + 5q^9 + \cdots,\\
	\lambda_6(q) &\coloneqq \sum_{n=0}^\infty \frac{(-1)^n q^n (q,q^2)_n}{(-q,q)_n} = 1 - q + 3q^2 - 5q^3 + 6q^4 - 7q^5 + 11q^6 - 16q^7 + 18q^8 + \cdots,\\
	2\mu_6(q) &\coloneqq 1 + \sum_{n=0}^\infty \frac{(-1)^n q^{n+1} (1+q^n) (q,q^2)_n}{(-q,q)_{n+1}} = 1 + 2q - 3q^2 + 4q^3 - 4q^4 + 6q^5 - 11q^6 + \cdots.
\end{align}

The 8th order mock theta functions $U_0(q), S_0(q), S_1(q)$, introduced by Gordon--McIntosh~\cite{GordonMcIntosh2000}, are defined as follows:
\begin{align}
	U_0(q) &\coloneqq \sum_{n=0}^\infty \frac{q^{n^2} (-q,q^2)_n}{(-q^4,q^4)_n} = 1 + q + q^2 + q^4 - q^6 + q^7 + q^9 + \cdots,\\
	S_0(q) &\coloneqq \sum_{n=0}^\infty \frac{q^{n^2} (-q,q^2)_n}{(-q^2, q^2)_n} = 1 + q + q^2 - q^3 + 2q^5 - q^7 + q^9 + \cdots, \\
	S_1(q) &\coloneqq \sum_{n=0}^\infty \frac{q^{n(n+2)} (-q,q^2)_n}{(-q^2, q^2)_n} = 1 + q^3 + q^4 - q^5 - q^6 + q^7 + 2q^8 + \cdots.
\end{align}

Finally, the 10th order mock theta functions $X_{10}(q)$ and $\chi_{10}(q)$, also found in Ramanujan's lost notebook~\cite{Choi1999}, \cite[Chapter 8]{LostV}, are given by
\begin{align}
	X_{10}(q) &\coloneqq \sum_{n=0}^\infty \frac{(-1)^n q^{n^2}}{(-q,q)_{2n}} = 1 - q + q^2 + q^4 - 2q^5 + q^6 - q^7 + q^8 - 2q^9 + \cdots,\\
	\chi_{10}(q) &\coloneqq \sum_{n=0}^\infty \frac{(-1)^n q^{(n+1)^2}}{(-q,q)_{2n+1}} = q - q^2 + q^3 - 2q^4 + 2q^5 - q^6 + 2q^7 - 3q^8 + 3q^9 + \cdots.
\end{align}

For each of the above mock theta functions $f(q)$, let $c_f(n)$ denote the $n$-th Fourier coefficient of $f(q)$.

% --------------------------------------------------------------------------
\section{Proof of \cref{thm:main-1} for level $N = 2^k$}\label{sec:3}
% --------------------------------------------------------------------------

We first show that $c_1(n) \equiv c_2(n) \equiv c_4(n) \equiv c_8(n) \equiv c_{16}(n) \pmod{2}$. The key claim in the subsequent proofs is the relation
\begin{align}\label{eq:binomial}
	(q^{2k})_\infty \equiv (q^k)_\infty^2 \pmod{2}
\end{align}
for any $k > 0$, which follows immediately from the binomial theorem. As shown in~\eqref{eq:j1-2}, we have 
\begin{align}\label{eq:j1-mod2}
	j_1(\tau) \equiv \frac{1}{q} \frac{1}{(q)_\infty^{24}} \equiv \frac{1}{q} \frac{1}{(q^8)_\infty^3} \pmod{2}.
\end{align}
Similarly, we obtain
\begin{align*}
	j_2(\tau) &= \frac{1}{q} \frac{(q)_\infty^{24}}{(q^2)_\infty^{24}} \equiv \frac{1}{q} \frac{(q)_\infty^{24}}{(q)_\infty^{48}} = \frac{1}{q} \frac{1}{(q)_\infty^{24}} \pmod{2},\\
	j_4(\tau) &= \frac{1}{q} \frac{(q)_\infty^8}{(q^4)_\infty^8} \equiv \frac{1}{q} \frac{(q)_\infty^8}{(q)_\infty^{32}} = \frac{1}{q} \frac{1}{(q)_\infty^{24}} \pmod{2},\\
	j_8(\tau) &= \frac{1}{q} \frac{(q)_\infty^4 (q^4)_\infty^2}{(q^2)_\infty^2 (q^8)_\infty^4} \equiv \frac{1}{q} \frac{(q)_\infty^{12}}{(q)_\infty^{36}} = \frac{1}{q} \frac{1}{(q)_\infty^{24}} \pmod{2},\\
	j_{16}(\tau) &= \frac{1}{q} \frac{(q)_\infty^2 (q^8)_\infty}{(q^2)_\infty (q^{16})_\infty^2} \equiv \frac{1}{q} \frac{(q)_\infty^{10}}{(q)_\infty^{34}} = \frac{1}{q} \frac{1}{(q)_\infty^{24}} \pmod{2},
\end{align*}
which yield the desired result. From this expression, \cref{thm:main-1} (1) is obvious. Next, we show (2).

\begin{proof}[Proof of (2)]
The first result, \eqref{eq:mu2-mod2}, was originally obtained by Wang~\cite[Theorem 3.3]{Wang2021}. However, to capture the essence of the argument, we provide a proof here. By using the Appell--Lerch series expression in~\cite[Equation~(2)]{McIntosh2007},
\[
	\mu_2(q) = 2\frac{(-q,q^2)_\infty}{(q^2)_\infty} \sum_{n \in \Z} \frac{(-q)^{\frac{n(n+1)}{2}}}{1+q^{2n}} = \frac{(-q,q^2)_\infty}{(q^2)_\infty} \left(1 + 2 \sum_{n \neq 0} \frac{(-q)^{\frac{n(n+1)}{2}}}{1+q^{2n}} \right),
\]
we obtain
\begin{align}\label{eq:mu2-cong}
	\mu_2(q) \equiv \frac{(-q, q^2)_\infty}{(q^2)_\infty} \pmod{2}.
\end{align}
Applying \eqref{eq:binomial}, we have
\begin{align}\label{eq:mu2-mod2}
	\mu_2(q) \equiv \frac{(q,q^2)_\infty}{(q)_\infty^2} = \frac{1}{(q)_\infty (q^2)_\infty} \equiv \frac{1}{(q)_\infty^3} \pmod{2}.
\end{align}
By comparing with \eqref{eq:j1-mod2}, we obtain
\[
	j_1(\tau) \equiv \frac{1}{q} \mu_2(q^8) \pmod{2},
\]
which implies $c_1(8n-1) \equiv c_{\mu_2}(n) \pmod{2}$.

As for the 8th order mock theta function \( U_0(q) \), it is evident from the definitions that
\[
    U_0(q) \equiv \mu_2(q) \pmod{2}.
\]
Wang~\cite[Theorem 8.6]{Wang2021} also established this equivalence by showing
\[
    U_0(q) \equiv \frac{1}{(q)_\infty^3} \equiv \mu_2(q) \pmod{2}.
\]
This completes the proof.

\end{proof}

\begin{proof}[Proof of (3) and (4)]
By the relation
\[
	U_0(q) = S_0(q^2) + q S_1(q^2)
\]
shown in~\cite[(1.7)]{GordonMcIntosh2000}, we find that $c_{S_0}(n) = c_{U_0}(2n)$ and $c_{S_1}(n) = c_{U_0}(2n+1)$ which yield the desired results.
\end{proof}

% --------------------------------------------------------------------------
\section{Proof of \cref{thm:main-2} for level $N = 3 \cdot 2^k$}
% --------------------------------------------------------------------------

We show that $c_3(n) \equiv c_6(n) \equiv c_{12}(n) \pmod{2}$. First, we establish the congruences for $j_3(\tau)$ and $j_{12}(\tau)$. Using \eqref{eq:binomial}, we have
\begin{align}\label{eq:j3-mod2}
	j_3(\tau) &= \frac{1}{q} \frac{(q)_\infty^{12}}{(q^3)_\infty^{12}} \equiv \frac{1}{q} \frac{(q^4)_\infty^3}{(q^{12})_\infty^3} \pmod{2},\\ \nonumber
	j_{12}(\tau) &= \frac{1}{q} \frac{(q^4)_\infty^4 (q^6)_\infty^2}{(q^2)_\infty^2 (q^{12})_\infty^4} \equiv \frac{1}{q} \frac{(q)_\infty^{16} (q^3)_\infty^4}{(q)_\infty^4 (q^3)_\infty^{16}} = \frac{1}{q} \frac{(q)_\infty^{12}}{(q^3)_\infty^{12}} \pmod{2}.
\end{align}
Next, we show that $j_3(\tau) \equiv j_6(\tau) \pmod{2}$. Since
\[
	j_6(\tau) = \frac{1}{q} \frac{(q^2)_\infty^3 (q^3)_\infty^9}{(q)_\infty^3 (q^6)_\infty^9} \equiv \frac{1}{q} \frac{(q)_\infty^6 (q^3)_\infty^9}{(q)_\infty^3 (q^3)_\infty^{18}} = \frac{1}{q} \frac{(q)_\infty^3}{(q^3)_\infty^9} \equiv \frac{1}{q} \frac{(q)_\infty^2}{(q^3)_\infty^6} \frac{(q)_\infty}{(q^3)_\infty^3} \pmod{2},
\]
applying the 2-dissection formula from~\cite[Lemma 2.1]{HirschhornRoselin2010},
\begin{align}\label{eq:2-dissect}
	\frac{(q)_\infty}{(q^3)_\infty^3} &= \frac{(q^2)_\infty (q^4)_\infty^2 (q^{12})_\infty^2}{(q^6)_\infty^7} - q \frac{(q^2)_\infty^3 (q^{12})_\infty^6}{(q^4)_\infty^2 (q^6)_\infty^9} \nonumber\\
		&\equiv \frac{(q)_\infty^{10}}{(q^3)_\infty^6} + q \frac{(q^3)_\infty^6}{(q)_\infty^2} \pmod{2},
\end{align}
we obtain
\[
	j_6(\tau) \equiv \frac{1}{q} \frac{(q)_\infty^2}{(q^3)_\infty^6} \left(\frac{(q)_\infty^{10}}{(q^3)_\infty^6} + q \frac{(q^3)_\infty^6}{(q)_\infty^2} \right) \equiv \frac{1}{q} \frac{(q)_\infty^{12}}{(q^3)_\infty^{12}} + 1 \pmod{2}.
\]
This confirms that $c_6(n) \equiv c_3(n) \pmod{2}$ for $n \ge 1$.

Now, we proceed with the proof of \cref{thm:main-2} (1). While most of (1) has already been shown by Ray~\cite{Ray2023}, the case for $c_3(24n+7)$ still remains to be addressed. Here, we present a simpler and more unified proof that covers both Ray's results and the remaining case, using two of the three Borwein--Borwein--Garvan's functions \cite{BBG1994} defined as
\begin{align*}
	a(q) \coloneqq \sum_{m,n \in \Z} q^{m^2 + mn+n^2}, \qquad c(q) \coloneqq \sum_{m,n \in \Z} q^{m^2 + mn + n^2 + m + n}.
\end{align*}

\begin{proof}[Proof of (1)]
From the expression in~\eqref{eq:j3-mod2}, we find that $c_3(n) \equiv 0 \pmod{2}$ whenever $n \not\equiv 3 \pmod{4}$ and
\[
	\sum_{n=0}^\infty c_3(4n-1) q^{4n-1} \equiv \frac{1}{q} \frac{(q^4)_\infty^3}{(q^{12})_\infty^3} \pmod{2},
\]
which gives
\[
	\sum_{n=0}^\infty c_3(4n-1) q^n \equiv \frac{(q)_\infty^3}{(q^3)_\infty^3} \pmod{2}.
\]
By applying the 3-dissection formula from \cite[Equations~(1.4) and (1.6)]{HirschhornGarvanBorwein1993},
\[
	\frac{(q)_\infty^3}{(q^3)_\infty} = a(q^3) - q c(q^3),
\]
we have
\[
	\sum_{n=0}^\infty c_3(4n-1) q^n \equiv \frac{1}{(q^3)_\infty^2} \bigg(a(q^3) + q c(q^3) \bigg) \pmod{2},
\]
which implies that $c_3(12n+7) \equiv 0 \pmod{2}$ and
\begin{align}\label{eq:c3-12+3}
	\sum_{n=0}^\infty c_3(12n-1) q^n \equiv \frac{a(q)}{(q)_\infty^2} \pmod{2}, \qquad \sum_{n=0}^\infty c_3(12n+3) q^n \equiv \frac{c(q)}{(q)_\infty^2} \pmod{2}.
\end{align}
By definition, all coefficients of positive powers of $q$ in $a(q)$ are even, implying $a(q) \equiv 1 \pmod{2}$. Consequently, the former congruence of \eqref{eq:c3-12+3} yields that $c_3(24n+11) \equiv 0 \pmod{2}$ and
\begin{align}\label{eq:c3-24-1}
	\sum_{n=0}^\infty c_3(24n-1)q^n \equiv \frac{1}{(q)_\infty} \pmod{2}.
\end{align}
Thus, we have proved all claims of (1).
\end{proof}

\begin{proof}[Proof of (2)]
By a simple congruence
\[
	\frac{(-1)^n (q, q^2)_n q^{n^2}}{(-q,q)_{2n}} = \frac{(-1)^n (q)_{2n} q^{n^2}}{(q^2)_n (-q,q)_{2n}} \equiv \frac{q^{n^2}}{(q)_n^2} \pmod{2}
\]
and Jacobi's identity
\[
	\sum_{n=0}^\infty \frac{q^{n^2}}{(q)_n^2} = \frac{1}{(q)_\infty} = 1 + \sum_{n=1}^\infty p(n) q^n,
\]
we have $\phi_6(q) \equiv 1/(q)_\infty \pmod{2}$. Comparing this with \eqref{eq:c3-24-1} implies $c_3(24n-1) \equiv c_{\phi_6}(n) \pmod{2}$.

We now complete the proof of (2) by referring to Wang's results. Wang~\cite[Theorem 4.1]{Wang2021} showed that
\[
	f_3(q) \equiv \phi_3(q) \equiv \frac{1}{(q)_\infty} \pmod{2}
\]
and that $c_{2\mu_6}(2n) \equiv p(n) \pmod{2}$ in \cite[Theorem 6.8]{Wang2021}.
\end{proof}

\begin{proof}[Proof of (3)]
By applying the 2-dissection formula found in~\cite[(1.36)]{HirschhornGarvanBorwein1993},
\[
	c(q) = 3 \frac{(-q^2, q^2)_\infty^3 (q^2)_\infty (q^6)_\infty}{(-q^6,q^6)_\infty} + qc(q^4) \equiv (q^2)_\infty^4 + q c(q^4) \pmod{2}
\]
to the latter congruence in \eqref{eq:c3-12+3}, we obtain
\[
	\sum_{n=0}^\infty c_3(12n+3) q^n \equiv (q^2)_\infty^3 + q \frac{c(q^4)}{(q^2)_\infty} \pmod{2},
\]
which implies that
\begin{align}\label{eq:c3-24+3}
	\sum_{n=0}^\infty c_3(24n+3) q^n \equiv (q)_\infty^3 \pmod{2}, \qquad \sum_{n=1}^\infty c_3(24n-9) q^n \equiv q \frac{c(q^2)}{(q)_\infty} \pmod{2}.
\end{align}
By the relation given in~\cite[(1.7)]{HirschhornGarvanBorwein1993},
\[
	c(q) = 3 \frac{(q^3)_\infty^3}{(q)_\infty},
\]
we obtain
\[
	\sum_{n=1}^\infty c_3(24n-9)q^n \equiv q \frac{(q^6)_\infty^3}{(q^2)_\infty (q)_\infty} \equiv q \frac{(q^3)_\infty^6}{(q)_\infty^3} \pmod{2}
\]
that is congruent to $\psi_6(q)$ modulo 2 by \cite[Theorem 6.2]{Wang2021}.
\end{proof}

\begin{proof}[Proof of (4)]
The claim immediately follows from comparing the former congruence in \eqref{eq:c3-24+3} with the well-known formula for Ramanujan's theta function
\begin{align}\label{eq:Jacobi-triple}
	\psi(q) \coloneqq \sum_{n=0}^\infty q^{\frac{n(n+1)}{2}} = (-q,q)_\infty^2 (q)_\infty \equiv (q)_\infty^3 \pmod{2}
\end{align}
as given by the Jacobi triple product identity.
\end{proof}

The claims (1) to (4) cover all arithmetic progressions. The last claim (5) is simply a corollary that immediately follows from Wang's result on the 6th order mock theta function $\lambda_6(q)$.

\begin{proof}[Proof of (5)]
Wang~\cite[Theorem 6.5]{Wang2021} showed that $c_{\lambda_6}(2n-1) \equiv c_{\psi_6}(n) \pmod{2}$ and the coefficient $c_{\lambda_6}(2n)$ is odd if and only if $n = k(k+1)/2$ for some $k \ge 0$.
\end{proof}

\begin{corollary}\label{cor-Ray}
We have 
\[
	\#\{0 \le n \le X : c_3(24n+3) \equiv 1 \pmod{2}\} = \left\lfloor \frac{1+\sqrt{1+8X}}{2} \right\rfloor,
\]
which in turn establishes Ray's Theorem $(4)$~\cite[Theorem 3.2]{Ray2023}.
\end{corollary}

\begin{proof}
By \cref{thm:main-2} (4), the number of $0 \le n \le X$ such that $c_3(24n+3) \equiv 1 \pmod{2}$ is equal to the number of $k \in \Z_{\ge 0}$ such that $\frac{k(k+1)}{2} \le X$.
\end{proof}

% --------------------------------------------------------------------------
\section{Proof of \cref{thm:main-3} for level $N = 5 \cdot 2^k$}
% --------------------------------------------------------------------------

To show the congruence $c_5(n) \equiv c_{10}(n) \pmod{2}$, we first recall the 2-dissection formula in \cite[(4.3)]{Ray2023},
\begin{align}\label{eq:level10-2-dissect}
	\frac{(q)_\infty}{(q^5)_\infty} &= \frac{(q^2)_\infty (q^8)_\infty (q^{20})_\infty^3}{(q^4)_\infty (q^{10})_\infty^3 (q^{40})_\infty} - q \frac{(q^4)_\infty^2 (q^{40})_\infty}{(q^8)_\infty (q^{10})_\infty^2} \nonumber \\
		&\equiv \frac{(q)_\infty^6}{(q^5)_\infty^2} + q (q^5)_\infty^4 \pmod{2}.
\end{align}
Using this, we find
\begin{align*}
	j_{10}(\tau) &= \frac{1}{q} \frac{(q^2)_\infty (q^5)_\infty^5}{(q)_\infty (q^{10})_\infty^5} \equiv \frac{1}{q} \frac{1}{(q^5)_\infty^4} \frac{(q)_\infty}{(q^5)_\infty} \equiv \frac{1}{q} \frac{1}{(q^5)_\infty^4} \left(\frac{(q)_\infty^6}{(q^5)_\infty^2} + q (q^5)_\infty^4 \right) \\
		&= \frac{1}{q} \frac{(q)_\infty^6}{(q^5)_\infty^6} + 1 = j_5(\tau) + 1 \pmod{2}.
\end{align*}

\begin{proof}[Proof of (1)]
Since
\[
	j_5(\tau) \equiv \frac{1}{q} \frac{(q^2)_\infty^3}{(q^{10})_\infty^3} \pmod{2},
\]
we see that $c_5(2n) \equiv 0 \pmod{2}$, and
\[
	\sum_{n=0}^\infty c_5(2n-1) q^n \equiv \frac{(q)_\infty^3}{(q^5)_\infty^3} \equiv \frac{(q^2)_\infty}{(q^{10})_\infty} \frac{(q)_\infty}{(q^5)_\infty} \pmod{2}.
\]
Using the 2-dissection formula \eqref{eq:level10-2-dissect} again, we obtain
\[
	\sum_{n=0}^\infty c_5(2n-1) q^n \equiv \frac{(q^2)_\infty}{(q^{10})_\infty} \left(\frac{(q)_\infty^6}{(q^5)_\infty^2} + q (q^5)_\infty^4\right) \equiv \frac{(q^4)_\infty^2}{(q^{20})_\infty} + q (q^2)_\infty (q^{10})_\infty \pmod{2}.
\]
The even-power and odd-power terms are given by
\begin{align}\label{eq:c5-4npm1}
	\sum_{n=0}^\infty c_5(4n-1) q^n \equiv \frac{(q^2)_\infty^2}{(q^{10})_\infty} \pmod{2}, \qquad \sum_{n=0}^\infty c_5(4n+1) q^n \equiv (q)_\infty (q^5)_\infty \pmod{2}.
\end{align}
From the former congruence, it follows that $c_5(8n+3) \equiv 0 \pmod{2}$ and
\[
	\sum_{n=0}^\infty c_{5}(8n-1) q^n \equiv \frac{(q)_\infty^2}{(q^5)_\infty} \equiv \frac{(q^2)_\infty}{(q^5)_\infty} \pmod{2}.
\]
By the 5-dissection formula obtained from \cite{Kang1999}, \cite[Theorem 7.4.1]{Berndt2006},
\[
	\frac{(q^2)_\infty}{(q^5)_\infty} = \frac{(q^{50})_\infty}{(q^5)_\infty} \left(\frac{1}{F(q^{10})} - q^2 - q^4 F(q^{10})\right),
\]
where
\[
	F(q) \coloneqq \frac{(q,q^5)_\infty (q^4, q^5)_\infty}{(q^2, q^5)_\infty (q^3, q^5)_\infty} = \cfrac{1}{1 + \cfrac{q}{1+ \cfrac{q^2}{\ddots}}},
\]
we have
\begin{align}\label{eq-level10-5dissect}
	\sum_{n=0}^\infty c_5(8n-1) q^n \equiv \frac{(q^{50})_\infty}{(q^5)_\infty} \left(\frac{1}{F(q^{10})} + q^2 + q^4 F(q^{10})\right) \pmod{2}.
\end{align}
This leads to $c_5(40n+7) \equiv c_5(40n+23) \equiv 0 \pmod{2}$, as derived from Ray's proof.

From the latter congruence in \eqref{eq:c5-4npm1} and using the 2-dissection formula in \cite[Lemma 3.2]{TangXia2020},
\begin{align*}
	(q)_\infty (q^5)_\infty &= \frac{(q^4)_\infty^2 (q^{10})_\infty^5}{(q^2)_\infty (q^5)_\infty^2 (q^{20})_\infty^2} - q \frac{(q^2)_\infty^5 (q^{20})_\infty^2}{(q)_\infty^2 (q^4)_\infty^2 (q^{10})_\infty}\\
		&\equiv (q^2)_\infty^3 +q (q^{10})_\infty^3 \pmod{2},
\end{align*}
we find
\begin{align}\label{eq:5-4-1}
	\sum_{n=0}^\infty c_5(8n+1) q^n \equiv (q)_\infty^3 \pmod{2}, \qquad \sum_{n=0}^\infty c_5(8n+5) q^n \equiv (q^5)_\infty^3 \pmod{2}.
\end{align}
The latter congruence implies that $c_5(40n+5+8i) \equiv 0 \pmod{2}$ for $i \in \{1,2,3,4\}$, and
\begin{align}\label{eq:5-4-2}
	\sum_{n=0}^\infty c_5(40n+5) q^n \equiv (q)_\infty^3 \pmod{2}.
\end{align}
Thus, all congruences in (1) are obtained.
\end{proof}

\begin{proof}[Proof of (2)]
From the part of \eqref{eq-level10-5dissect} containing the powers that are multiples of 5, we have
\begin{align*}
	\sum_{n=0}^\infty c_5(40n-1) q^n &\equiv \frac{(q^{10})_\infty}{(q)_\infty} \frac{1}{F(q^2)} = \frac{(q^{10})_\infty}{(q)_\infty} \frac{(q^4, q^{10})_\infty (q^6, q^{10})_\infty}{(q^2, q^{10})_\infty (q^8, q^{10})_\infty} \\
	&= \frac{(q^{10})_\infty^2 (q^4, q^{10})_\infty^2 (q^6,q^{10})_\infty^2}{(q)_\infty (q^2)_\infty} \equiv \frac{(q^{20})_\infty (q^8, q^{20})_\infty (q^{12}, q^{20})_\infty}{(q)_\infty^3} \pmod{2}.
\end{align*}
By Wang's result in~\cite[Theorem 9.2]{Wang2021}, it is congruent to the 10th order mock theta function $X_{10}(q)$ modulo 2.
\end{proof}

\begin{proof}[Proof of (3)]
Similarly, from the part of \eqref{eq-level10-5dissect} containing the powers that are congruent to 4 modulo 5, we have
\begin{align*}
	\sum_{n=1}^\infty c_5(40n-9) q^n &\equiv q \frac{(q^{10})_\infty}{(q)_\infty} F(q^2) = q \frac{(q^{10})_\infty}{(q)_\infty} \frac{(q^2, q^{10})_\infty (q^8, q^{10})_\infty}{(q^4, q^{10})_\infty (q^6, q^{10})_\infty}\\
		&= q \frac{(q^{10})_\infty^2 (q^2, q^{10})_\infty^2 (q^8, q^{10})_\infty^2}{(q)_\infty (q^2)_\infty} \equiv q \frac{(q^{20})_\infty (q^4, q^{20})_\infty (q^{16}, q^{20})_\infty}{(q)_\infty^3} \pmod{2}.
\end{align*}
By Wang's result~\cite[Theorem 9.2]{Wang2021} again, it is known that this is congruent to the 10th order mock theta function $\chi_{10}(q)$ modulo 2.
\end{proof}

\begin{proof}[Proof of (4)]
By the first congruence in~\eqref{eq:5-4-1}, \eqref{eq:5-4-2}, and \eqref{eq:Jacobi-triple}, we have
\[
	\sum_{n=0}^\infty c_5(8n+1) q^n \equiv \sum_{n=0}^\infty c_5(40n+5) q^n \equiv (q)_\infty^3 \equiv \sum_{n=0}^\infty q^{\frac{n(n+1)}{2}} \pmod{2}.
\]
This directly establishes the desired result.
\end{proof}

\begin{proof}[Proof of (5)]
From the part of \eqref{eq-level10-5dissect} containing the powers that are congruent to 2 modulo 5, we have
\[
	\sum_{n=0}^\infty c_5(40n+15) q^n \equiv \frac{(q^{10})_\infty}{(q)_\infty} \pmod{2}.
\]
The right-hand side is the generating series for $p_{10}(n)$, which counts the number of partitions $n$ where no parts are multiples of 10.
\end{proof}

% --------------------------------------------------------------------------
\section{Concluding remarks}\label{sec:remarks}
% --------------------------------------------------------------------------

Congruences among hauptmoduln are observed not only modulo 2, but also modulo 3, 5, 7, and 13.

\begin{proposition}
	For $n \ge 1$, we have the following:
	\begin{enumerate}
		\item $c_1(n) \equiv c_3(n) \equiv c_9(n) \pmod{3}$.
		\item $c_2(n) \equiv c_6(n) \equiv c_{18}(n) \pmod{3}$.
		\item $c_4(n) \equiv c_{12}(n) \pmod{3}$.
		\item $c_1(n) \equiv c_5(n) \equiv c_{25}(n) \pmod{5}$.
		\item $c_2(n) \equiv c_{10}(n) \pmod{5}$.
		\item $c_1(n) \equiv c_7(n) \pmod{7}$.
		\item $c_1(n) \equiv c_{13}(n) \pmod{13}$.
	\end{enumerate}
\end{proposition}

\begin{proof}
The claims (1) and (4) immediately follow from the fact that $j_1(\tau) \equiv \frac{1}{q(q)_\infty^{24}} \pmod{p}$ and that $(q^{pk})_\infty \equiv (q^k)_\infty^p \pmod{p}$ for $p \in \{3,5\}$. For claim (2), we obtain
\begin{align*}
	j_6(\tau) &= \frac{1}{q} \frac{(q^2)_\infty^3 (q^3)_\infty^9}{(q)_\infty^3 (q^6)_\infty^9} \equiv \frac{1}{q} \frac{(q^2)_\infty^3 (q)_\infty^{27}}{(q)_\infty^3 (q^2)_\infty^{27}} = \frac{1}{q} \frac{(q)_\infty^{24}}{(q^2)_\infty^{24}} = j_2(\tau) \pmod{3},\\
	j_{18}(\tau) &= \frac{1}{q} \frac{(q^6)_\infty (q^9)_\infty^3}{(q^3)_\infty (q^{18})_\infty^3} \equiv \frac{1}{q} \frac{(q^2)_\infty^3 (q)_\infty^{27}}{(q)_\infty^3 (q^2)_\infty^{27}} = \frac{1}{q} \frac{(q)_\infty^{24}}{(q^2)_\infty^{24}} = j_2(\tau) \pmod{3}.
\end{align*}

For claim (3), we have
\begin{align*}
%	j_4(\tau) &= \frac{1}{q} \frac{(q)_\infty^8}{(q^4)_\infty^8},\\
	j_{12}(\tau) &= \frac{1}{q} \frac{(q^4)_\infty^4 (q^6)_\infty^2}{(q^2)_\infty^2 (q^{12})_\infty^4} \equiv \frac{1}{q} \frac{(q^4)_\infty^4 (q^2)_\infty^6}{(q^2)_\infty^2 (q^4)_\infty^{12}} = \frac{1}{q} \frac{(q^2)_\infty^4}{(q^4)_\infty^8} \pmod{3}.
\end{align*}
By the 2-dissection formula in \cite[(34.1.18)]{Hirschhorn2017}, 
\[
	\frac{(q^3)_\infty^3 (q^4)_\infty (q^{12})_\infty}{(q)_\infty (q^6)_\infty^2} = \frac{(q^4)_\infty^4}{(q^2)_\infty^2} + q \frac{(q^{12})_\infty^4}{(q^6)_\infty^2},
\]
we can show that
\[
	\frac{(q)_\infty^8}{(q^4)_\infty^8} \equiv \frac{(q^2)_\infty^4}{(q^4)_\infty^8} + q \pmod{3},
\]
which implies $j_4(\tau) \equiv j_{12}(\tau) + 1 \pmod{3}$.

For claim (5), we have
\[
	j_{10}(\tau) = \frac{1}{q} \frac{(q^2)_\infty (q^5)_\infty^5}{(q)_\infty (q^{10})_\infty^5} \equiv \frac{1}{q} \frac{(q^2)_\infty (q)_\infty^{25}}{(q)_\infty (q^2)_\infty^{25}} = \frac{1}{q} \frac{(q)_\infty^{24}}{(q^2)_\infty^{24}} = j_2(\tau) \pmod{5}.
\]

For the last claims (6) and (7), we recall the equations
\begin{align*}
	\frac{E_4(\tau)^3 - 3 \cdot 691 E_{12}(\tau)}{\Delta(\tau)} &= -2072 j_1(\tau) + 1296000 = - \frac{2072}{q} - 245568 + O(q),\\
	\frac{E_4(\tau)^3 - 7 \cdot 691 E_{12}(\tau)}{\Delta(\tau)} &= -4836j_1(\tau) + 3024000 = -\frac{4836}{q} - 573984 + O(q),
\end{align*}
where
\[
	E_{12}(\tau) = 1 + \frac{65520}{691} \sum_{n=1}^\infty \bigg(\sum_{d \mid n} d^{11}\bigg) q^n
\]
is the weight 12 Eisenstein series. These can be proved by comparing the first two Fourier coefficients, as both sides are (weakly holomorphic) modular forms of weight $0$, which take the form of a polynomial in $j_1$. They imply that
\[
	j_1(\tau) - \frac{1}{q} \frac{1}{(q)_\infty^{24}} \equiv 6 \pmod{7}, \qquad j_1(\tau) - \frac{1}{q} \frac{1}{(q)_\infty^{24}} \equiv 5 \pmod{13},
\]
that is, $j_1(\tau) \equiv j_7(\tau) + 6 \pmod{7}$ and $j_1(\tau) \equiv j_{13}(\tau) + 5 \pmod{13}$. The first congruences of (1) and (4) can also be shown in the same form. The key point here is that $65520 = 2^4 \cdot 3^2 \cdot 5 \cdot 7 \cdot 13$.
\end{proof}

Even when focusing solely on these congruences between hauptmoduln, they appear non-trivial based on their definition as generators of the field of modular functions. In some cases, proving these congruences requires non-trivial dissection formulas, further indicating that the relationships are not obvious.

Additionally, various congruences modulo primes for the coefficients of mock theta functions can also be investigated as summarized in \cite{Andersen2014}. For instance, by comparing the 3rd order mock theta function $f_3(q)$ defined in \eqref{def:f3} with Ramanujan's another 3rd order mock theta function
\begin{align}
	\chi_3(q) \coloneqq \sum_{n=0}^\infty \frac{q^{n^2} (-q,q)_n}{(-q^3, q^3)_n} = 1 + q + q^2 + q^6 + q^7 - q^{10} + \cdots,
\end{align}
we immediately obtain $c_{f_3}(n) \equiv c_{\chi_3}(n) \pmod{3}$ from the congruence
\[
	\frac{(-q,q)_n}{(-q^3, q^3)_n} \equiv \frac{(-q,q)_n}{(-q,q)_n^3} = \frac{1}{(-q,q)_n^2} \pmod{3}.
\]
As a natural question, we may ask whether similar comprehensive relationships between the coefficients of hauptmoduln or weakly holomorphic modular forms and odd-order mock theta functions can also be found modulo odd primes. In addition, \cref{thm:main-3} (5) is currently a compromise result that lacks uniformity seen in our other results. Below are the first terms of $n$ up to 30 such that $c_5(40n+15) \equiv 1 \pmod{2}$.
\[
	n = 0, 1, 3, 4, 5, 6, 7, 10, 11, 12, 15, 18, 21, 22, 23, 24, 25, 28, 29, \dots.
\]
Is there a relationship between the coefficients $c_5(40n+15)$ and the coefficients of mock theta functions?

%\appendix
% --------------------------------------------------------------------------
%\section{List of $n$ such that $c_f(n) \equiv 1 \pmod{2}$}\label{sec:list}
% --------------------------------------------------------------------------

%\bibliographystyle{amsplain}
\bibliographystyle{amsalpha}
\bibliography{References} 

\end{document}